\documentclass[12pt]{amsart}
\usepackage{amsmath,amssymb,amsfonts,amsthm,amsopn,bbm}
\usepackage{graphics}
\usepackage{graphicx,tikz}

 \def\Spnr{Sp(d,\R)}

\newcommand{\stft}{short-time Fourier transform}

\newtheorem{theorem}{Theorem}[section]
\newtheorem{lemma}[theorem]{Lemma}

\newtheorem{definition}[theorem]{Definition}

\newtheorem{remark}[theorem]{Remark}

\newcommand{\beqa}{\begin{eqnarray*}}
\newcommand{\eeqa}{\end{eqnarray*}}

\newcommand{\field}[1]{\mathbb{#1}}
\newcommand{\bR}{\field{R}}        
\newcommand{\bZ}{\field{Z}}        
\newcommand{\bC}{\field{C}}        
\def\la{\lambda}

\def\eps{\epsilon}

\def\cF{\mathcal{F}}              
\def\cS{\mathcal{S}}

\def\rd{\bR^n}

\def\rdd{{\bR^{2n}}}

\def\lrd{L^2(\rd)}

\def\intrd{\int_{\rd}}
\def\intrdd{\int_{\rdd}}

\def\R{\right)}

\def\<{\left<}
\def\>{\right>}

\def\mv1{M_v^1}

\def\phas{(x,\xi )}
\def\mn{(m,n)}
\def\mn'{(m',n')}

\def\Spnr{Sp(n,\R)}

\def\o{\xi}

\def\R{\mathbb{R}}
\def\Ren{\mathbb{R}^n}
\def\Renn{\mathbb{R}^{2n}}

\def\sch{\mathcal{S}}

\def\H{{\mathbb H}}

\def\f{\varphi}

\def\Sn2{S_{2}(L^{2}(\Ren))}
\def\S1{S_{1}(L^{2}(\Ren))}
\def\sig00{\sigma_{0,0}}

\def\la{\langle}
\def\ra{\rangle}

\begin{document}
\begin{abstract} This note contains a new characterization of modulation spaces $M^p(\rd)$, $1\leq p\leq \infty$, by symplectic rotations. Precisely, instead to measure the time-frequency content of a function by using translations and modulations of a fixed window as building blocks, we use translations and metaplectic operators corresponding to symplectic rotations. Technically, this amounts to replace, in the computation of the $M^p(\rd)$-norm, the integral in the time-frequency plane with an integral on $\rd\times U(2n,\bR)$ with respect to a suitable measure, $U(2n,\bR)$ being the group of symplectic rotations. More conceptually, we are considering a sort of polar coordinates in the time-frequency plane.  In this new framework, the Gaussian invariance under symplectic rotations yields to choose Gaussians as suitable window functions. We also provide a similar characterization with the group $U(2n,\bR)$ being reduced to the $n$-dimensional torus $\mathbb{T}^n$. 
\end{abstract}

\title[A characterization of modulation spaces by symplectic rotations]{A characterization of modulation spaces by symplectic rotations}

\author{Elena Cordero}
\address{Dipartimento di Matematica, Universit\`a di Torino, Dipartimento di
	Matematica, via Carlo Alberto 10, 10123 Torino, Italy}
\email{elena.cordero@unito.it}
\thanks{}
\author{Maurice de Gosson}
\address{University of Vienna, Faculty of Mathematics,
	Oskar-Morgenstern-Platz 1 A-1090 Wien, Austria}
\email{maurice.de.gosson@univie.ac.at}
\thanks{}
\author{Fabio Nicola}
\address{Dipartimento di Scienze Matematiche, Politecnico di Torino, corso
	Duca degli Abruzzi 24, 10129 Torino, Italy}
\email{fabio.nicola@polito.it}


\maketitle\section{Introduction}
The objective of this study is to find  a new characterization of modulation spaces using symplectic rotations. Precisely, we are interested in those metaplectic operators $\widehat{S}\in Mp(n,\R)$, such that the corresponding projection $S:=\pi(\widehat{S})$ onto the symplectic group $\Spnr$ is a symplectic rotation. Let us recall that the symplectic group $\Spnr$ is the subgroup of $2n\times 2n$ invertible matrices $GL(2n,\bR)$,   defined by
\begin{equation}\label{sympM}
\Spnr=\left\{S\in GL(2n,\R):\;S J S^T=J\right\},
\end{equation}
where $J$ is the orthogonal matrix
$$
J=\begin{pmatrix} 0_n&I_n\\-I_n&0_n\end{pmatrix},
$$
($I_n$, $0_n$ are  the $n\times n$ identity matrix and null matrix, respectively).
Here we consider the subgroup $$U(2n,\bR):= \Spnr\cap O(2n,\bR)\simeq U(n)$$ of symplectic rotations (cf., e.g. \cite[Section 2.3]{Gos11}), namely
\begin{equation}\label{sympS}
U(2n,\bR)=\left\{\begin{pmatrix} A&-B\\B&A\end{pmatrix} :\;AA^T+BB^T=I_n,\, AB^T=B^T A\right\}\subset \Spnr,
\end{equation}
endowed with the normalized Haar measure $dS$ (the group $U(2n,\bR)$, being compact,  is unimodular).

In the 80's H. Feichtinger \cite{F1} introduced modulation spaces to measure the time-frequency concentration of a function/distribution on the time-frequency space (or phase space) $\rdd$.  They are nowadays become popular among mathematicians and engineers because they have found numerous applications in signal processing \cite{CFP2015,F2,F3}, pseudodifferential and Fourier integral operators \cite{wiener9,wiener8,Wiener,Toft04,Toftweight},  partial differential equations \cite{benyi,benyi2,benyi3,benyi4, fio5, fio3,  EFsurvey,EFsurvey, Wangbook, baoxiang,baoxiang2} and quantum mechanics \cite{MetapWiener13,Gos11}.

To recall their definition, we need a few time-frequency tools. First, the translation $T_x$ and modulation $M_\o$ operators are defined by
$$T_x f(t)=f(t-x),\quad M_\o f(t)=e^{2\pi i t\cdot \o} f(t), \quad t,x,\o\in \rd,
$$
for any function $f$ on $\rd$.

The time-frequency representation which occurs in the definition of modulation spaces is the short-time Fourier Transform (STFT)
of a distribution $f\in\cS'(\rd)$ with respect to a function $g\in\cS(\rd)\setminus\{0\}$ (so-called window), given by
\begin{equation}\label{stftdef}
V_g f\phas=\la f,M_\o T_xg\ra=\intrd f(t)\overline{g(t-x)}\,e^{-2\pi i t\cdot \o}dt,\quad x,\o\in\rd.
\end{equation}
The  \stft\ is well-defined whenever  the bracket $\langle \cdot , \cdot \rangle$ makes sense for
dual pairs of function or distribution spaces, in particular for $f\in
\cS ' (\rd )$, $g\in \cS (\rd )$,  or for $f,g\in\lrd$.\par

\begin{definition}[\bf Modulation spaces]  \label{prva}
Given  $g\in\cS(\rd)$, and $1\leq p\leq
\infty$, the {\it
  modulation space} $M^{p}(\Ren)$ consists of all tempered
distributions $f\in \cS' (\rd) $ such that $V_gf\in L^{p}(\Renn )$. The norm on $M^{p}(\rd)$ is
\begin{equation}\label{defmod}
\|f\|_{M^{p}}=\|V_gf\|_{L^{p}}=\left(\int_{\Renn}
  |V_gf(x,\o)|^p
    dxd\o \right)^{1/p}=\left(\int_{\Renn}
    |\la f,M_\o T_xg\ra |^p
    dxd\o \right)^{1/p}
\end{equation}
(with obvious modifications for $p=\infty$).
\end{definition}
The spaces $M^{p}(\rd)$ are Banach spaces,  and
 every nonzero $g\in M^{1}(\rd)$ yields an equivalent norm in
 \eqref{defmod}, so that their definition is  independent of the choice
 of $g\in  M^{1}(\rd)$ (see \cite{F1,book}).

We now provide an equivalent norm to \eqref{defmod} by using translations $T_x$ (or modulations $M_\xi$) and 
the operators $\widehat{S}$, with $S\in U(2n,\bR)$ as follows.
\begin{theorem}\label{main}  
	Consider the Gaussian function $\f(t)=2^{d/4} e^{-\pi |t|^2}$. 
	
	(i) For $1\leq p<\infty$ and  $f\in M^p(\rd)$, we have 
	\begin{equation}\label{normeq}
	\|f\|_{M^p(\rd)}\asymp\left(\int_{\bR^n\times  U(2n,\bR)} |x|^n |\la f, \widehat{S}T_x \f\ra|^p dx dS\right)^{\frac{1}{p}},
	\end{equation}
	where $dx$ is the Lebesgue measure on $\rd$ and $dS$ the Haar measure on $U(2n,\bR)$.\par
	 Similarly,
	\begin{equation}\label{normeqxi}
	\|f\|_{M^p(\rd)}\asymp\left(\int_{\bR^n\times  U(2n,\bR)} |\xi|^n |\la f, \widehat{S}M_\xi \f\ra|^p d\xi dS\right)^{\frac{1}{p}},
	\end{equation}
	with $d\xi$ being the Lebesgue measure on $\rd$ and $dS$ the Haar measure on $U(2n,\bR)$.
	
	(ii) For $p=\infty$, $f\in M^\infty(\rd)$, it occurs
	\begin{equation}\label{normeqinf}
	\|f\|_{M^\infty (\rd)}\asymp \sup_{S\in U(2n,\bR)} \sup_{x\in \rd} |\la f, \widehat{S}T_x \f\ra|
	\end{equation}
	or, similarly,
	\begin{equation}\label{normeqinffreq}
	\|f\|_{M^\infty (\rd)}\asymp \sup_{S\in U(2n,\bR)} \sup_{\xi\in \rd} |\la f, \widehat{S} M_\xi \f\ra|.
	\end{equation}
	\end{theorem}

The interpretation of the integral \eqref{normeq} above is as follows. The metaplectic operator $\widehat{S}$ produces a time-frequency rotation of the shifted Gaussian $T_x\f$. In this way, the operator 
$$f\mapsto\langle f,\widehat{S}T_x\varphi\rangle
$$ 
detects the time-frequency content of $f$ in an oblique strip, see Figure \ref{figura}. All the contributions are then added together with a weight $|x|^n$ which takes into account the underlapping of the strips as $|x|\to \infty$ and the overlapping as $|x|\to 0$.\par

Formulas \eqref{normeqxi}, \eqref{normeqinf} and \eqref{normeqinffreq} have similar meanings.
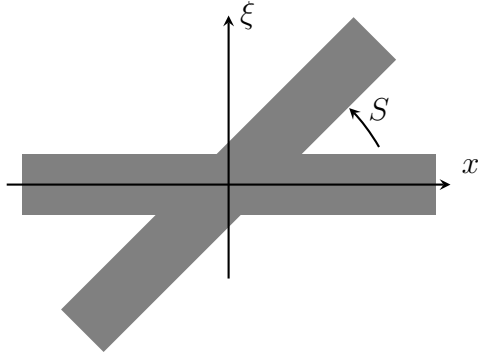
\begin{figure}[htb]\label{figura}
	\centering
	\begin{tikzpicture}[xscale=0.5,yscale=0.5]
	\fill [gray] (-5.5,-0.8) rectangle (5.5,0.8);
	\fill [gray, rotate=45] (-5.5,-0.8) rectangle (5.5,0.8);
	\draw[-stealth,thick] (-5.9,0) -- (5.9,0) node[above right] {$x$};
	\draw[-stealth,thick] (0,-2.5) -- (0,4.5) node[right] {$\xi$};
	\draw[-stealth,thick] (4,1) arc (30:45:5);
	\node at (4,2) {$S$};
	\end{tikzpicture}
	\caption{The time-frequency content of $f$ in the oblique strip is detected by the operator $f\mapsto\langle f,\widehat{S}T_x\varphi\rangle$}.
\end{figure}

Observe that in dimension $n=1$, $U(2,\bR)\simeq U(1)$ and the above formula is essentially a transition to polar coordinates with $|x|$ being the Jacobian.

Comparing \eqref{defmod} and \eqref{normeq} we observe that in \eqref{normeq} the  modulation operator $M_\xi$ is replaced by the metaplectic operator $\widehat{S}$ and the integral on the phase space $\rdd$ has become an integral on the cartesian product $\rd\times U(2n,\bR)$. The integration parameters $\phas$ of \eqref{defmod} live in $\rdd$, with dim $\rdd=2n$, whereas the parameters $(x,S)$ of \eqref{normeq} live in $\rd\times U(2n,\bR)$. Recall that dim $U(2n,\bR)=n^2$ \cite{Gos11};   this suggests  that a formula similar to \eqref{normeq}   should hold when $U(2n,\bR)$ is reduced to a suitable subgroup  $K\subset U(2n,\bR)$ of dimension $n$. This is indeed the case, as shown in the subsequent Theorem \ref{main2}. 

Consider the $n$-dimensional torus 
\begin{equation}\label{torus}
\mathbb{T}^n =\left \{S=\begin{pmatrix} e^{i\theta_1}& & \\& &\ddots& \\
& & & e^{i\theta_n}\end{pmatrix}\,:\, \theta_1,\dots,\theta_n\in \bR\right\}\subset U(n)
\end{equation}
with the Haar measure $dS=d\theta_1\dots d\theta_n$. The torus is isomorphic to a subgroup $K\subset U(2n, \bR)$, via the isomorphism $\iota$ in formula \eqref{iota} below (see the subsequent section). 

We exhibit the following characterization for $M^p$-spaces.
\begin{theorem}\label{main2} 	Let $\f$ be the Gaussian of Theorem \ref{main}.\par
(i)	For $1\leq p<\infty$, $f\in M^p(\rd)$, we have
	\begin{equation}\label{ctorus}
	\|f\|_{M^p(\rd)}\asymp \left(\int_{\rd\times \mathbb{T}^n } |x_1\dots x_n| |\la f, \widehat{S}T_x \f \ra |^p dx dS\right)^{\frac 1 p},
	\end{equation}
	or, similarly,
	\begin{equation}\label{ctorusfreq}
	\|f\|_{M^p(\rd)}\asymp \left(\int_{\rd\times \mathbb{T}^n } |\xi_1\dots \xi_n| |\la f, \widehat{S}M_\xi \f \ra |^p d\xi dS\right)^{\frac 1 p}.
	\end{equation}
	(ii) For $p=\infty$, 
	\begin{equation}\label{ctorusinf}
	\|f\|_{M^\infty(\rd)}\asymp \sup_{{S}\in \mathbb{T}^n } \sup_{x\in\rd}|\la f, \widehat{S}T_x \f \ra |
	\end{equation}
	or 
	\begin{equation}\label{ctorusinffreq}
	\|f\|_{M^\infty(\rd)}\asymp \sup_{{S}\in \mathbb{T}^n } \sup_{\xi\in\rd}|\la f, \widehat{S}M_\xi \f \ra |.
	\end{equation}
\end{theorem}
The above results for the groups $U(2n,\bR)$ and $\mathbb{T}^n$ can be interpreted, in a sense, as two extreme cases, and it would be interesting to find, more generally, for which compact subgroups $K\subset U(2n,\bR)$ similar characterizations hold. We conjecture that they should be precisely the subgroups $K\subset U(2n,\bR)$ such that every orbit for their action on $\mathbb{R}^{2n}$ intersects $\{0\}\times\mathbb{R}^n$ (up to subsets of measure zero), with a corresponding weighted measure on $\mathbb{R}^{n}\times K$ to be determined. \par

Another open problem which is worth investigating is the study of discrete versions of the above characterizations.\par\medskip
The paper is organized as follows: in Section \ref{sec2} we collected some preliminary results, whereas Section \ref{sec3} is devoted to the proof of Theorems \ref{main} and \ref{main2}. In Section 4 we rephrase more explicitly Theorem \ref{main2} in terms of the partial fractional Fourier transform.

\section{Notation and Preliminaries}\label{sec2}
\textbf{Notation.} We write $x\cdot y$ for  the scalar product on
$\Ren$ and $|t|^2=t\cdot t$, for $t,x,y \in\Ren$. For expressions $A,B\geq 0$, we
use the notation
$A\lesssim B$ to represent the inequality
$A\leq c B$ for a suitable
constant $c>0$, and  $A
\asymp B$  for the equivalence  $c^{-1}B\leq
A\leq c B$.

The Schwartz class is denoted by
$\sch(\Ren)$, the space of tempered
distributions by  $\sch'(\Ren)$.   We
use the brackets  $\la f,g\ra$ to
denote the extension to $\sch '
(\Ren)\times\sch (\Ren)$ of the inner
product $\la f,g\ra=\int f(t){\overline
	{g(t)}}dt$ on $L^2(\Ren)$.\par\bigskip
{\bf Metaplectic Operators.} The metaplectic representation $\mu$ of $Mp(n,\bR)$, the two-sheeted cover of the symplectic group $\Spnr$, defined in  \eqref{sympM} arises as intertwining operator between the
standard Schr\"odinger representation $\rho$ of the Heisenberg
group $\H^d$ and the representation that is obtained from it by
composing $\rho$ with the action of $\Spnr$ by automorphisms on
$\H^d$ (see, e.g., \cite{Gos11, folland89, leray}). 
Let us recall  the main points of a direct construction. \par 
The symplectic group $\Spnr$ is generated by the so-called free symplectic matrices \[S=\left(\begin{matrix}A & B \\ C & D\end{matrix}\right)\in \Spnr,\quad \det B\neq 0.\] To each such a matrix the associated generating function is defined by \[W(x,x')=\frac{1}{2}DB^{-1}x\cdot x-B^{-1}x\cdot x'+\frac{1}{2}B^{-1}Ax'\cdot x'.\] Conversely, to every polynomial of the type \[W(x,x')=\frac{1}{2}Px\cdot x-Lx\cdot x'+\frac{1}{2}Qx'\cdot x'\] with \[P=P^T, Q=Q^T\] and \[\det L\neq 0\] it can be associated a free symplectic matrix, namely \[S_W=\left(\begin{matrix}L^{-1}Q & L^{-1} \\ PL^{-1}Q-L^T & PL^{-1}\end{matrix}\right).\] 
Given $S_W$ as above and $m\in\mathbb{Z}$ such that 
\[
m\pi \equiv {\rm arg}\,\det L\quad {\rm mod}\, 2\pi,
\]
the related operator $\widehat{S}_{W,m}$ is defined by setting, for $\psi \in \mathcal{S}(\mathbb{R}^n)$,
\begin{equation}\label{opfree}
\widehat{S}_{W,m}\psi(x)=\frac{1}{i^{n/2}}\Delta(W)\int_{\mathbb{R}^n}e^{2\pi iW(x,x')}\psi(x')dx'
\end{equation}
(with $ i^{n/2}=e^{i\pi n/4}$) where
\[\Delta(W)=i^m\sqrt{|\det L|}.\]
The operator $\widehat{S}_{W,m}$ is named {\it quadratic Fourier transform} associated to the free symplectic matrix $S_W$.
The class modulo 4 of the integer $m$ is called {\it Maslov index} of $\widehat{S}_{W,m}$. Observe that if $m$ is one choice of Maslov index, then $m+2$ is another equally good choice: hence to each function $W$ we associate two operators, namely $\widehat{S}_{W,m}$ and $\widehat{S}_{W,m+2}=-\widehat{S}_{W,m}$.

The quadratic Fourier transform corresponding to the choices $S_W=J$ and $m=0$ is denoted by $\widehat{J}$. The generating function of $J$ is simply $W(x,x')=-x\cdot x'$. It follows that
\begin{equation}\label{Jmeta}
\widehat{J}\psi(x)=\frac{1}{ i^{n/2}}\int_{\mathbb{R}^n}e^{-2\pi ix\cdot x'}\psi(x')dx'=\frac{1}{i^{n/2}}\cF\psi(x)
\end{equation}for $\psi\in \mathcal{S}(\mathbb{R}^n)$, where $\cF$ is the usual unitary Fourier transform.\par
The quadratic Fourier transforms $\widehat{S}_{W,m}$ form a subset of the group $\mathcal{U}(L^2(\mathbb{R}^n))$ of unitary operators acting on $L^2(\mathbb{R}^n)$, which is closed under the operation of inversion and they generate a subgroup of $\mathcal{U}(L^2(\mathbb{R}^n))$ which is, by definition,
the metaplectic group $Mp(n,\mathbb{R})$. The elements of $Mp(n,\mathbb{R})$ are called metaplectic operators.
\par
Hence, every $\widehat{S}\in Mp(n,\mathbb{R})$ is, by definition, a product \[\widehat{S}_{W_1,m_1}\ldots\widehat{S}_{W_k,m_k}\] of metaplectic operators associated to free symplectic matrices.\par
Indeed, it can be proved that every $\widehat{S}\in Mp(n,\mathbb{R})$ can be written as a product of exactly two quadratic Fourier transforms: $\widehat{S}=\widehat{S}_{W,m}\widehat{S}_{W',m'}$. 
Now, it can be shown that the mapping 
\[
\widehat{S}_{W,m}\longmapsto S_W
\]
extends to a group homomorphism 
\[
\pi: Mp(n,\mathbb{R})\to Sp(n,\mathbb{R}),
\]
which is in fact a double covering. \par
We also observe that each metaplectic operator is, by construction, a unitary operator in $L^2(\rd)$, but also an automorphism of $\mathcal{S}(\rd)$ and of $\mathcal{S}'(\rd)$.

We are interested in its restriction $\widehat{S}=\pi(S)$, with $S\in U(2n,\bR)$, the symplectic rotations in \eqref{sympS}.

Observe that $U(n):=U(n,\bC)$, the complex unitary group (the group of $n\times n$ invertible complex matrices $V$ satisfying $VV^\ast=V^\ast V=I_n$) is isomorphic to $U(2n,\bR)$. The isomorphism $\iota$ is the mapping $ \iota: U(n)\to U(2n, \bR)$ given by
\begin{equation}\label{iota}
\iota(A+iB)=\begin{pmatrix} A&-B\\B&A\end{pmatrix},
\end{equation}
for details see \cite[Chapter 2.3]{Gos11}.

We present here some results related to the group $U(2n,\bR)$, which will be used in the sequel to attain the characterization of Theorem \ref{main}.
First, we recall a well-known result, see for instance \cite[Lemma 9.4.3]{book}:
\begin{lemma}\label{Lemma00}
	For $f,g\in\lrd$ and $S \in \Spnr$, the STFT $V_g f$ satisfies
	\begin{equation}\label{C1}
|V_{\widehat{S}g} ({\widehat{S}f})\phas|=|V_g f (S^{-1}\phas)|,\quad \phas\in\rdd.
	\end{equation}
\end{lemma}
This second issue is contained in \cite{CNT17}, we sketch the proof for the sake of consistency.
\begin{lemma}\label{Lemma0}
For $\f,\psi\in \cS(\rd)$ and $S\in U(2n,\bR)$, the STFT $V_\f  (\widehat{S}\psi)$ is a Schwartz function, with seminorms uniformly bounded when $S\in U(2n,\bR)$.
\end{lemma}
\begin{proof}
Since $\f \in \cS(\rd)$, the STFT $V_\f$ is a continuous mapping from $ \cS(\rd)$ into $\cS(\rdd)$ (see \cite{F1}). Hence, it is enough to show that
\[\{\hat{S}\varphi:\,S\in U(2n,\mathbb{R})\}\] is a bounded subset of the Schwartz class $\mathcal{S}(\mathbb{R}^n)$, i.e.,  every Schwartz seminorm is bounded on it. Since the group  $U(2n,\mathbb{R})$ is compact, it is sufficient to show that every seminorm is locally bounded, that is, we can limit ourselves to consider $S$ in a sufficiently small neighbourhood for any fixed $S_0\in U(2n,\mathbb{R})$. Equivalently, we can consider $S$ of the form $S=S_1J^{-1}S_0$ where $S_1$ belongs to a enough small neighbourhood of $J$ in $U(2n,\R)$. Using the representation of metaplectic operators recalled at the beginning of this section, we can write
\begin{align*}
\hat{S}\varphi(x)&=\pm \widehat{S}_1[\widehat{J}^{-1}\widehat{S}_0\varphi](x)\\
&=c \sqrt{|\det L|}\int_{\mathbb{R}^n}e^{2\pi i(\frac{1}{2}Px\cdot x-Lx\cdot y+\frac{1}{2}Qy\cdot y)}[\underbrace{\widehat{J}^{-1}\widehat{S}_0\varphi]}_{\in \mathcal{S}(\mathbb{R}^n)}(y)dy
\end{align*}
where $|c|=1$ and, we might say, $\|P\|<\epsilon$, $\|Q\|<\epsilon$, $\|L-I\|<\epsilon$. If $\epsilon<1$, it is straightforward to check that  $\hat{S}\varphi$ belongs to a bounded subset of $\mathcal{S}(\mathbb{R}^n)$, as desired.
\end{proof}
\begin{lemma}\label{lemma1}
Let $B=(b_{i,j})_{i,j=1,\dots n}$ be the $n\times n$ submatrix in \eqref{sympS}.	The subset $\Sigma\subset U(2n,\bR)$ obtained by setting $b_{i,1}=0,$ $i=1,\dots n$ (i.e., the first column of $B$ is set to zero), is a submanifold of codimension $n$.
\end{lemma}
\begin{proof}
We have to verify that the coordinates $b_{1,1},\dots, b_{n,1}$ are independent on the subset $\Sigma$, namely the projection
$$(b_{1,1},\dots,b_{n,1}): U(2n,\bR)\to \rd
$$
has rank $n$ on $\Sigma$.

Let us first show that for every $S_0\in \Sigma$ there exists a $U(2n,\bR)$-valued smooth function $S(b_{1},\dots,b_{n})$, defined in a neighbourhood of $0\in \rd$, such that $S(0)=S_0$ and the first column ``of its submatrix $B$" is precisely $(b_{1},\dots,b_{n})^T$. 
	
	Let $S_0= A+i B=(V_1,\dots,V_n)\in \Sigma$, with $V_j$ being a $n\times 1$ complex vector, $j=1,\dots, n$, so that by assumption $(b_{i,1})_{i=1,\dots, n}=\textrm{Im}\, V_1=0$. We consider any smooth function $V_1(b_{1},\dots,b_{n})$, defined in a neighbourhood of $0\in \rd$, valued in the unit sphere of $\bC^n$, such that
	$$\textrm{Im}\, V_1(b_{1},\dots,b_{n})=(b_{1},\dots,b_{n})^T,\quad V_1(0)=V_1.
	$$
	Then, we apply the Gram-Schmidt orthonormalization procedure in $\bC^n$ to the set of vectors $(V_1(b_{1},\dots,b_{n}),V_2,\dots,V_n)$. This provides the desired $U(n)$-valued function
	$S(b_{1},\dots,b_{n})$. In particular $S(0)=S_0$.\par
	Now, the composition of the mapping
	$$(b_{1},\dots,b_{n})\mapsto S(b_{1},\dots,b_{n})
	$$
followed by the projection $(b_{1,1},\dots,b_{n,1}): U(2n,\bR)\to \rd$ is therefore the identity mapping in a neighbourhood of $0$ and has rank $n$. Hence the same is true for the projection $(b_{1,1},\dots,b_{n,1}): U(2n,\bR)\to \rd$ at $S_0$.
\end{proof}
\begin{lemma}\label{lemma2}
For every $\eps>0$, define \begin{equation}\label{ki}
\chi_\eps\phas=\frac{1}{\eps^n} \mathbbm{1}_{Q}\left(\frac{\xi}{\eps}\right),
\end{equation}
where 
$$Q=\left[-\frac 12,\frac12\right]^n\subset \rd \quad \mbox{and}\,\quad \mathbbm{1}_{Q}=\begin{cases}
1, \,\,\,\xi \in Q
\\
0, \,  \,\, \xi \notin Q\\
\end{cases} 
$$
and 
\begin{equation}\label{kitilde}\displaystyle
\tilde{\chi}_\eps(z)=\frac{\chi_\eps(z)}{\int_{U(2n,\bR)}\chi_\eps(Sz) \,dS}, \quad z\in\rdd.
\end{equation}
Then we have
\begin{equation}\label{PU}
\int_{U(2n,\bR)}\tilde{\chi}_\eps(Sz) \,dS=1,\quad \forall z\in \rdd
\end{equation}
and 
\begin{equation}\label{CD}
\lim_{\eps\to 0^+}\int_{\rdd}\tilde{\chi}_\eps\phas\Phi\phas dx d\xi=C\int_{\rd} |x|^n \Phi(x,0)\, dx,
\end{equation}
for some $C>0$ and for every continuous function $\Phi$ on $\rdd$ with a rapid decay at infinity.
\end{lemma}
\begin{proof}
We will show in a moment that, for $z=\phas\in\rdd$,
\begin{equation}\label{P1}
\int_{{U(2n,\bR)}}\chi_\eps(Sz)\,dS \gtrsim \min \{\eps^{-n}, {|z|^{-n}}\}
\end{equation}
(with the convention, at $z=0$, that $\min \{\eps^{-n}, +\infty\}=\epsilon^{-n}$). In particular, $\int_{U(2n,\bR)}\chi_\eps(Sz)\,dS\,\not= 0,$ for every $z\in\rdd$. Formula \eqref{PU} then follows, because 
\begin{align*}
\int_{U(2n,\bR)}\tilde{\chi}_\eps(Sz) \,dS&=\int_{U(2n,\bR)}\frac{{\chi}_\eps(Sz)}{\int_{U(2n,\bR)} \chi_\eps(U S z) \,dU} dS\\
&= \int_{U(2n,\bR)}\frac{{\chi}_\eps(Sz)}{\int_{U(2n,\bR)} \chi_\eps(U  z) \,dU} dS=1
\end{align*}
for every $z\in\rdd$, since the Haar measure is right invariant. \par
Let us now prove \eqref{P1}. For $z=0$ we have 
$$\int_{U(2n,\bR)}{\chi}_\eps(Sz) \,dS=\frac{1}{\eps^n}\int_{U(2n,\bR)} dS=\frac{C_0}{\eps^n},
$$
with $C_0=meas(U(2n,\bR))>0$. Consider now $z\not=0$. Observe that the function 
$$\Psi_\eps(z):=\int_{U(2n,\bR)}{\chi}_\eps(Sz) \,dS
$$
is constant on the orbits of $U(2n,\bR)$ in $\rdd$, so that we can suppose
$$z=(x,0),\quad x=(x_1,0,\dots,0), \quad x_1=|x|=|z|>0.
$$
Now, by the definition of $\chi_\eps$ and $\Psi_\eps$,
\begin{equation}\label{P2}
\Psi_\eps(z)=\eps^{-n} \, meas \left\{S=\begin{pmatrix} A&-B\\B&A\end{pmatrix}\in U(2n,\bR): \, |b_{i,1}|<\frac{\eps}{2|z|},\ i=1,\dots,n \right\},
\end{equation}
where $(b_{i,1})_{i=1,\dots, n}$, is the first column of the matrix $B=(b_{i,j})_{i,j=1,\dots n}$.\par Define, for $\mu>0$, 
$$f(\mu)=meas \left\{ S=\begin{pmatrix} A&-B\\B&A\end{pmatrix}\in U(2n,\bR): \, |b_{i,1}|<\mu,\ i=1,\dots,n \right\}. 
$$
Observe that $f(\mu)$ is non-decreasing and constant for $\mu\geq 1$. Moreover, from Lemma \ref{lemma1} we know that by setting $b_{i,1}=0$, $i=1,\dots,n$, in $U(2n,\bR)$, we get a submanifold $\Sigma$ of codimension $n$, and the function $f(\mu)$ is the measure of a tubular neighbourhood of $\Sigma$ in  $U(2n,\bR)$. Hence we have the asymptotic behaviour 
\begin{equation}\label{P3}
\mu^{-n}f(\mu) \to C_0>0,\quad \mbox{as}\quad \mu\to 0^+
\end{equation}
and in particular
\begin{equation}\label{P4}
f(\mu)\gtrsim \min \{1,\mu^n\}.
\end{equation}
We then infer 
\begin{equation}\label{P5}
\Psi_\eps(z)=\eps^{-n} f\left(\frac{\eps}{2 |z|}\right)\to \frac{C_1}{|z|^n},\quad \mbox{as} \ \eps\to 0^+
\end{equation}
locally uniformly in $\rdd\setminus\{0\}$, with $C_1=2^{-n}C_0$, and
\begin{equation}\label{P6}
\Psi_\eps(z)\gtrsim \eps^{-n}\min\left\{1,\left(\frac{\eps}{|z|}\right)^n\right\}=\min\{\epsilon^{-n},|z|^{-n}\},
\end{equation}
which is \eqref{P1}.

Let us finally prove \eqref{CD}. We are interested in the limit $\eps\to 0^+$, so we can assume $\eps\leq 1$. Consider a continuous function $\Phi$ on $\rdd$ with rapid decay at infinity. By definition of $\tilde{\chi}_\eps(z)$ in \eqref{kitilde} we have
$$\tilde{\chi}_\eps\phas=\frac{\eps^{-n}}{\Psi_\eps \phas} \mathbbm{1}_{[-\eps/2, \eps/2]^n}(\xi)
$$ 
so that, by \eqref{P6},
$$ |\tilde{\chi}_\eps\phas\Phi\phas|\lesssim \eps^{-n}(1+|x|^n) \mathbbm{1}_{[-\eps/2, \eps/2]^n}(\xi)|\Phi\phas|\in L^1(\rdd)$$ 
for $0<\eps\leq 1$. Fubini's Theorem then allows one to  look at the first integral in \eqref{CD}  as an iterated integral
$$
I_\eps:=\int_{\rd} \left(\int_{\rd}\tilde{\chi}_\eps\phas\Phi\phas d\xi\right) dx
$$
and we apply the dominated convergence theorem to the integral with respect to the $x$ variable as follows. 
Setting 
$$\Upsilon_\eps(x):=\int_{\rd}\tilde{\chi}_\eps\phas\Phi\phas d\xi=\eps^{-n}\int_{[-\eps/2, \eps/2]^n}\frac{1}{\Psi_\eps\phas}\Phi\phas\,d\xi,
$$
by \eqref{P5} we have, for every fixed $x\not=0$, 
\[
\Upsilon_\eps(x)\to C |x|^n \Phi(x,0);
\]
for some constant $C>0$. On the other hand $\Upsilon_\eps(x)$ is dominated, using \eqref{P6}, by
$$(1+|x|)^n \sup_{\xi\in\rd}|\Phi\phas|\in L^1(\rd).
$$
Hence
$$\lim_{\eps\to 0^+} I_\eps=\intrd \lim_{\eps\to 0^+} \Upsilon_\eps(x) dx=C\int_{\rd} |x|^n \Phi(x,0)\, dx.
$$
This concludes the proof.
\end{proof}

\begin{remark}
	Observe that there are no conditions on the derivatives of the function $\Phi$ in \eqref{CD}.
\end{remark}

\section{Proofs of the main results}\label{sec3} In what follows we prove Theorems \ref{main} and \ref{main2}.

\begin{proof}[Proof of Theorem \ref{main}.] (i) \textbf{First Step.} Let us start with showing that formula \eqref{normeq} is true for any function $\psi$ in the Schwartz class $\cS(\rd)	\subset M^p(\rd)$, $1\leq p <\infty$. Using the Gaussian $\f(t)=2^{d/4} e^{-\pi |t|^2}$ as window function, we compute the $M^p$-norm of  $\psi$ as in \eqref{defmod} and then use Lemma \ref{lemma2} so that
	\begin{align*}
	\|\psi\|_{M^p}^p&=\intrdd |V_\f \psi(z) |^p\, dz=\intrdd \int_{U(2n,\bR)} \tilde{\chi}_\eps (Sz) |V_\f \psi (z)|^p \, dS dz\\&=
	\intrdd \int_{U(2n,\bR)} \tilde{\chi}_\eps (z) |V_\f \psi (S^{-1}z)|^p \, dS dz\\
	&=	\intrdd \int_{U(2n,\bR)} \tilde{\chi}_\eps (z) |V_{\widehat{S}\f} \widehat{S}\psi (z)|^p \, dS dz\\
	\end{align*}
where in the last equality we used Lemma \ref{Lemma00}. Observe that, since $S$ is unitary and $\f$ is a Gaussian, $\widehat{S}\f=c\f$, for some phase factor $c\in \bC$, with $|c|=1$ (see \cite[Proposition 252]{Gos11}) and this phase factor is killed by the modulus obtaining 
	$|V_{\widehat{S}\f} \widehat{S}\psi (z)|=|V_\f \widehat{S}\psi (z)|$. Continuing the above computation we infer
	$$	\|\psi\|_{M^p}^p= \intrdd \tilde{\chi}_\eps (z) \int_{U(2n,\bR)}  |V_{\f} \widehat{S}\psi (z)|^p \, dS dz.
	$$
	Set \[
	\Phi(z)=\int_{U(2n,\bR)}  |V_{\f} \widehat{S}\psi (z)|^p \, dS.
	\] 
The dominated convergence theorem guarantees that $\Phi$ is continuous on $\rdd$, moreover $\Phi$ has rapid decay at infinity. This follows from Lemma \ref{Lemma0}. 

	Letting $\eps\to 0^+$ and using \eqref{CD}  we obtain
\begin{align*}
\|\psi\|_{M^p}^p&=C\intrd |x|^n \int_{U(2n,\bR)} |V_\f \widehat{S}\psi(x,0)|^p\,dS dx\\ &=C\intrd |x|^n \int_{U(2n,\bR)} |\la \widehat{S}\psi, T_x \f\ra |^p\,dS dx\\
& =C\intrd |x|^n \int_{U(2n,\bR)} |\la \psi, \widehat{S}T_x \f\ra |^p\,dS dx.
\end{align*}
The last equality is due to $\la \widehat{S}\psi ,T_x \f \ra =\la \psi ,\widehat{S}^{-1} T_x \f \ra$ and  the invariance of the Haar measure with respect to the change of variable ${S} \to {S}^{-1}$. \par\medskip
\textbf{Second Step.} Consider $f \in M^p(\rd)$, $1\leq p<\infty$. Using the density of the Schwartz class $\cS(\rd)$ in $M^p(\rd)$ (cf.\ e.g., \cite[Chapter 12]{book}), there exists a sequence $\{\psi_k\}_k\in \cS(\rd)$ such that $\psi_k\to f$ in $M^p(\rd)$. This implies that $\psi_k\to f$ in $\cS'(\rd)$ and 
$$ \la \psi_k, \widehat{S}T_x \f\ra \to \la \psi, \widehat{S}T_x \f\ra
$$
pointwise for every $x\in\rd$, $S\in U(2n,\bR)$. Let us define, for every  $f\in M^p(\rd)$,
\begin{equation}\label{normatil}
|||f|||=\left(\int_{\bR^n\times  U(2n,\bR)} |x|^n |\la f, \widehat{S}T_x \f\ra|^p dx dS\right)^{\frac{1}{p}}.
\end{equation}
By Fatou's Lemma, for any $f\in M^p(\rd)$:
\begin{equation}\label{stima}
|||f|||^p\leq \liminf_{k\to \infty}|||\psi_k|||^p\lesssim \liminf_{k\to \infty}\|\psi_k\|_{{M}^p}^p=\|f\|_{{M}^p}^p.
\end{equation}
It is easy to check that $|||f|||$ is a seminorm on $M^p(\rd)$. Applying \eqref{stima} to the difference $f-\psi_k$ we obtain $|||f-\psi_k|||\to 0$  and hence $|||\psi_k|||\to |||f|||$. By assumption we also have $\|\psi_k\|_{{M}^p} \to \|f\|_{{M}^p}$, and the desired norm equivalence in \eqref{normeq} then extents from $\cS(\rd)$ to $M^p(\rd)$. \par\medskip
\textbf{Third Step.} We will show that \eqref{normeqxi} easily follows from \eqref{normeq}.  Indeed, the Fourier transform $\widehat{J}=\cF$ is a metaplectic operator
and we recall that the Fourier transform is a topological isomorphism $\cF: M^p(\rd) \to 
 M^p(\rd)$, $1\leq p\leq \infty$, \cite{F1}.  Furthermore, by the definition of the symplectic group \eqref{sympM}, for any $S\in U(2n,\bR)$, 
 $$ J^{-1}S=(S^{T})^{-1}J^{-1}=S J^{-1}
 $$
for $S^{-1}=S^{T}$.  For any $f\in M^p(\rd)$, $\|f\|_{M^p}\asymp \|\hat{f}\|_{M^p}$, and using \eqref{Jmeta},
\begin{align*}
|\la \hat{f}, S T_x\f\ra|&=|\la f , \widehat{J^{-1}}\widehat{S}T_x\f\ra|= |\la f , \widehat{S}\cF^{-1}T_x\f\ra|\\
&= |\la f , \widehat{S} M_x \cF^{-1}{\f}\ra|=|\la f , \widehat{S} M_x {\f}\ra|
\end{align*}
since the Gaussian is an eigenvector of $\cF^{-1}$ with eigenvalue equal to $1$. This immediately yields \eqref{normeqxi}.

(ii) Case $p=\infty$. Observe that any $z\in\rdd$ can be written as
$$ z=S^{-1}\begin{pmatrix} x\\0 \end{pmatrix}, 
$$
for some $x\in\rd$, $S\in U(2n,\bR)$, so that, for any $f\in M^\infty(\rd) $,
\begin{align*}
\|f\|_{M^\infty(\rd)}&=\sup_{z\in\rdd}|V_\f f(z)|\asymp \sup_{S\in U(2n,\bR)}\sup_{ x\in \rd}\left| V_\f f\left(S^{-1}\begin{pmatrix} x\\0 \end{pmatrix} \right)\right|\\
&=\sup_{S\in U(2n,\bR)}\sup_{ x\in \rd}| V_\f (S f)(x,0)|=\sup_{S\in U(2n,\bR)}\sup_{ x\in \rd}|\la S f, T_x \f\ra|\\
&= \sup_{S\in U(2n,\bR)}\sup_{ x\in \rd}|\la  f, S T_x \f\ra|,
\end{align*}
as desired.
\end{proof} 

We now prove the similar result, with the group $U(2n,\bR)$ replaced by the subgroup  $\mathbb{T}^n$ (up to isomorphisms).

\begin{proof}[Proof of Theorem \ref{main2} ]
(i) The proof uses a similar pattern to Theorem \ref{main}, replacing the group $U(2n,\bR)$ by $\mathbb{T}^n$. The preparation of Lemma \ref{lemma1} is no longer necessary.  Lemma \ref{lemma2} must be revisited in this context as follows. Using the same notation,  to estimate the function $\Psi_\eps(z)$ we again observe that it is constant on the  orbits of the torus $\mathbb{T}^n$, so that we can suppose 
$$z=(x,0), \quad x=(x_1,\dots,x_n), \quad  x_j=|x_j|=|z_j|>0,\quad j=1,\dots,n.
$$ 
Then 
\begin{align*}
\Psi_\eps (z) &=\eps^{-n} meas \left\{ S=\begin{pmatrix} e^{i\theta_1}& & \\& &\ddots& \\
& & & e^{i\theta_n}\end{pmatrix}\,:\, |\sin \theta_j|< \frac{\eps}{2|z_j|}, \, j=1,\dots,n\right\}\\
&= \eps^{-n}f\left(\frac{\eps}{2|z_1|}\right)\dots f\left(\frac{\eps}{2|z_n|}\right),
\end{align*} 
where now we set, for $\mu>0$,
$$f(\mu)= meas \{\theta\in [0,2\pi]\,:\, |\sin \theta|<\mu\}.
$$
We have 
$$\mu^{-1} f(\mu) \to C>0\quad \mbox{as}\quad \eps\to 0^+
$$
and 
$$ f(\mu)\gtrsim \min \{1,\mu\}
$$
which gives 
$$\Psi_\eps (z)\to \frac{C}{|z_1\cdots z_n|}, \quad \mbox{as} \quad \eps\to 0^+$$
locally uniformly for $z_1,\dots, z_n\in \bR^2\setminus\{0\}$, and 
$$\Psi_\eps (z)\gtrsim \min\{\eps^{-1}, |z_1|^{-1}\}\cdots \min\{\eps^{-1}, |z_n|^{-1}\}.
$$
Using these estimates in place of \eqref{P5} and \eqref{P6}, one can proceed as in the proofs of Lemma \ref{lemma2} and Theorem \ref{main} and obtain the desired conclusion for $f\in\cS(\rd)$. Next, using the same argument as in the Step $2$ of the previous proof, one infers \eqref{ctorus}. 

The characterization \eqref{ctorusfreq} has the same proof as the corresponding formula \eqref{normeqxi}.

(ii) The $M^\infty$ case uses the same argument as in the proofs of \eqref{normeqinf} and \eqref{normeqinffreq}, with the group $U(2n,\bR)$ replaced by $\mathbb{T}^n$.
\end{proof}
\section{Integral representations for the torus in terms of the factional Fourier transform}

Observe that the symplectic matrix in $U(2n,\bR)$ corresponding to the complex matrix $S\in\mathbb{T}^n$ in \eqref{torus} via the isomorphism $\iota$ in \eqref{iota} is given by
$$\iota(S)=\begin{pmatrix} A&-B \\B&A \end{pmatrix}
$$
with
$$A=\begin{pmatrix} \cos \theta_1& & \\& &\ddots& \\
& & & \cos \theta_n\end{pmatrix},\quad B= \begin{pmatrix} \sin \theta_1& & \\& &\ddots& \\
& & & \sin \theta_n\end{pmatrix}.
$$ 
Consider the case $\theta_i\not= k\pi$,  $k\in\bZ$, $i=1,\dots,n$. The matrix $\iota(S)$ is a free symplectic matrix and the related metaplectic operator possesses the integral representation \eqref{opfree}. Since 
$$AB^{-1}=B^{-1}A=\begin{pmatrix} \frac{\cos \theta_1}{\sin\theta_1}& & \\& &\ddots& \\
& & & \frac{\cos \theta_n}{\sin\theta_n}\end{pmatrix},
$$ 
the polynomial $W(x,x')$ becomes
\begin{equation}\label{WB}
W(x_1,\dots,x_n,x_1',\dots,x_n')=\sum_{i=1}^n \frac{1}{2\sin \theta_i}(\cos\theta_i x_i^2-2x_i x_i'+\cos\theta_i x_i'^2)
\end{equation}
and
$$\Delta(W)=\frac{c}{\sqrt{|\sin\theta_1\cdots\sin\theta_n|}}.$$ 
for some phase factor $c\in\bC$, with $|c|=1$. Hence we obtain, for $\psi\in\cS(\rd)$, 
\begin{equation}\label{iS}
\widehat{\iota(S)}\psi(x)=\frac{c}{\sqrt{|\sin\theta_1\cdots\sin\theta_n|}}\intrd e^{2\pi i W(x,x')}\psi (x') dx',
\end{equation}
with $W(x,x')$ in \eqref{WB}. From \eqref{iS} we deduce that $\widehat{\iota(S)}$ can be written as the composition of the operators \begin{equation}\label{prodS}
\widehat{\iota(S)}=\pm\widehat{\iota(S_1)}\cdots \widehat{\iota(S_n)},
\end{equation}
where, for some phase factor $c$, 
\[\label{iSi}\widehat{\iota(S_i)}\psi(x)=\frac{c}{\sqrt{|\sin\theta_i|}}\int_{\bR} e^{\frac{\pi i}{\sin \theta_i}(\cos\theta_i x_i^2-2x_i x_i'+\cos\theta_i x_i'^2)}\psi(x_1',\dots,x_i',\dots,x_n') dx_i'.
\]
Indeed if $\theta_i=\pi/2$, then $\widehat{\iota(S_i)}=\pm \widehat{J}$ is the Fourier transform with respect to the variable $x_i$. Otherwise, $\widehat{\iota(S_i)}=\pm\cF_{\theta_i}$, the $\theta_i$-angle partial fractional Fourier transform (again referred to the variable $x_i$).\par 

Alternatively, the same conclusion \eqref{prodS} can be drawn by writing 
\begin{equation}\label{productS}
S=\begin{pmatrix} e^{i\theta_1}& & \\& &\ddots& \\
& & & e^{i\theta_n}\end{pmatrix}=\begin{pmatrix} e^{i\theta_1}& & &\\& 1& & \\& &\ddots& \\
& & & 1\end{pmatrix}\cdots\begin{pmatrix} 1& & &\\&\ddots & &\\ & & 1 &  \\
& & & e^{i\theta_n}\end{pmatrix}
\end{equation}
that is
$$S=S_1\cdots S_i\cdots S_n,$$
with 
$$ S_i = \begin{pmatrix} 1& & & & & &\\&\ddots & & & & &\\& &1 & & & &\\
& & & e^{i\theta_i} & & &\\& & & &1& &\\ & & & & & \ddots & \\& & & & & &1\end{pmatrix},\quad i=1,\dots,n
$$
so that $$\widehat{\iota (S)}=\widehat{\iota ({S_1})\dots\iota ({S_1})}=\pm \widehat{\iota(S_1)}\cdots \widehat{\iota(S_n)}.$$
If $\theta_i= 2k\pi$ for some $k\in\bZ$,  $\widehat{\iota(S_i)}=\pm {I}$ with $I$ the identity operator. If $\theta_i= (2k+1)\pi$ for some $k\in\bZ$,  $\widehat{\iota(S_i)}\psi(x)=\pm \psi(x_1,\ldots,-x_i,\ldots,x_n) $.


Hence using the $\theta_i$-angle partial fractional Fourier transform $\cF_{\theta_i}=\pm\widehat{\iota(S_i)}$ we can rephrase Theorem \ref{main2} as follows.

\begin{theorem}\label{main2integral} 	Let $\f$ be the Gaussian of Theorem \ref{main}.\par
(i)	For $1\leq p<\infty$, $f\in M^p(\rd)$, we have
	\begin{equation}\label{ctorus}
	\|f\|_{M^p(\rd)}\asymp \left(\int_{\rd\times \mathbb{T}^n } |x_1\dots x_n| |\la f, \cF_{\theta_1}\dots \cF_{\theta_n} T_x \f \ra |^p dx d\theta_1\dots d\theta_n\right)^{\frac 1 p},
	\end{equation}
	or, similarly,
	\begin{equation}\label{ctorusfreq}
	\|f\|_{M^p(\rd)}\asymp \left(\int_{\rd\times \mathbb{T}^n } |\xi_1\dots \xi_n| |\la f, \cF_{\theta_1}\dots \cF_{\theta_n}M_\xi \f \ra |^p d\xi d\theta_1\dots d\theta_n\right)^{\frac 1 p}.
	\end{equation}
	(ii) For $p=\infty$, 
	\begin{equation}\label{ctorusinf}
	\|f\|_{M^\infty(\rd)}\asymp \sup_{{S}\in \mathbb{T}^n } \sup_{x\in\rd}|\la f, \cF_{\theta_1}\dots \cF_{\theta_n}T_x \f \ra |
	\end{equation}
	or 
	\begin{equation}\label{ctorusinffreq}
	\|f\|_{M^\infty(\rd)}\asymp \sup_{{S}\in \mathbb{T}^n } \sup_{\xi\in\rd}|\la f, \cF_{\theta_1}\dots \cF_{\theta_n}M_\xi \f \ra |.
	\end{equation}
\end{theorem}

This concludes our study. \par For sake of completeness, let us recall that integral representations involving metaplectic operators that do not arise from free symplectic matrices were studied in \cite{CNRintrep,MO1999}.

\section*{Acknowledgment}
The first and the third author have been supported by the Gruppo
Nazionale per l'Analisi Matematica, la Probabilit\`a e le loro
Applicazioni (GNAMPA) of the Istituto Nazionale di Alta Matematica
(INdAM). MdG has been financed by the Austrian Research Foundation FWF grant P27773.

\end{document}